\documentclass[11pt, a4paper]{amsart}
\usepackage[utf8]{inputenc}
\usepackage{amsmath}
\usepackage{amsfonts}
\usepackage{amssymb}
\usepackage{graphicx}
\usepackage{amsmath,amssymb, amsthm}
\usepackage{pgf}
\usepackage{tikz-cd}
\usepackage{xy}
\usepackage{hyperref}
\usetikzlibrary{matrix,arrows,decorations.pathmorphing}
\usepackage[english]{babel}
\usepackage{mathtools}
\usepackage{caption}
\usepackage{geometry}
\usepackage{float}

\theoremstyle{definition}

\theoremstyle{definition}
\newtheorem{rmk}{Remark}
\theoremstyle{definition}

\theoremstyle{definition}\usepackage{amsmath}

\theoremstyle{plain}
\newtheorem{thm}{Theorem}
\theoremstyle{plain}
\newtheorem{prop}{Proposition}
\theoremstyle{plain}

\theoremstyle{plain}
\newtheorem{lem}{Lemma}

\newtheorem{coro}{Corollary}

\newcommand{\Gal}{\operatorname{Gal}}

\newcommand{\Z}{\mathbb{Z}}
\newcommand{\F}{\mathbb{F}}
\newcommand{\Q}{\mathbb{Q}}

\newcommand{\C}{\mathbb{C}}

\newcommand{\Fp}{\mathbb{F}_p}

\newcommand{\Ok}{\mathcal{O}_K}
\newcommand{\Of}{\mathcal{O}_F}

\newcommand{\zk}{\zeta_K}

\newcommand{\Nm}{\operatorname{N}}

\newcommand{\lp}{\left(}
\newcommand{\rp}{\right)}
\newcommand{\N}{\mathbb{N}}
\newcommand{\pp}{\mathfrak{p}}

\newcommand{\Ol}{\mathcal{O}_L}

\newcommand{\Fq}{\mathbb{F}_q}

\makeatletter
\@namedef{subjclassname@2020}{%
  \textup{2020} Mathematics Subject Classification}
\makeatother

\author[F.~Battistoni]{Francesco Battistoni}
\address{Laboratoire de math\'{e}matiques de Besan\c{c}on\\
         Universit\'{e} Bourgogne Franche-Comt\'{e}\\
         CNRS - UMR 6623\\
         16\\ Route de Gray\\
         25030 Besan\c{c}on\\
         France}
\email{francesco.battistoni@univ-fcomte.fr}

\author[H.~Oukhaba]{Hassan Oukhaba}
\address{Laboratoire de math\'{e}matiques de Besan\c{c}on\\
         Universit\'{e} Bourgogne Franche-Comt\'{e}\\
         CNRS - UMR 6623\\
         16\\ Route de Gray\\
         25030 Besan\c{c}on\\
         France}
\email{hassan.oukhaba@univ-fcomte.fr}

\keywords{Arithmetic equivalence, global function fields, inverse Galois problem}

\subjclass[2020]{11M38, 11R58, 12F12}

\title{Arithmetic equivalence for non-geometric extensions of global function fields}

\begin{document}

\maketitle

\begin{abstract}
    In this paper we study couples of finite separable extensions of the function field $\F_q(T)$ which are arithmetically equivalent, i.e. such that prime ideals of $\Fq[T]$ decompose with the same inertia degrees in the two fields, up to finitely many exceptions. In the first part of this work, we extend previous results by Cornelissen, Kontogeorgis and Van der Zalm to the case of non-geometric extensions of $\Fq(T)$, which are fields such that their field of constants may be bigger than $\Fq$. In the second part, we explicitly produce examples of non-geometric extensions of $\F_2(T)$ which are equivalent and non-isomorphic over $\F_2(T)$ and non-equivalent over $\F_4(T)$, solving a particular Inverse Galois Problem.
\end{abstract}

\section{Introduction}
\noindent
Consider a number field $K$: every prime number $p\in\Z$ generates an ideal in the ring of integers $\Ok\subset K$ which uniquely decomposes as 
$$p\Ok = \mathfrak{p}_1^{e_1}\cdots\mathfrak{p}_r^{e_r}$$
where the ideals $\pp_i$ are prime ideals in $\Ok$. If $f_i$ is the inertia degree of $\pp_i$ over $p$, assuming $f_1\leq f_2\leq \cdots \leq f_r$, we define the \textit{splitting type of $p$ in $K$} to be the $r$-ple $f_K(p)\coloneqq (f_1,\ldots,f_r)$.

We say that two number fields $K$ and $L$ are \textit{arithmetically equivalent} if $f_K(p)=f_L(p)$ for every prime $p\in\Z$ up to a set of exceptions of Dirichlet density zero. This is an equivalence relation among number fields which is trivially satisfied if $K$ and $L$ are isomorphic; however, the converse is not always true, as it was already noted by Gassmann \cite{gassmann1926bemerkungen} who gave the first explicit example of two equivalent number fields which are not isomorphic. A more systematic study of arithmetic equivalence was carried several years later by Perlis \cite{perlis1977equation}, who showed that this concept is in fact equivalent to a relation, called \textit{Gassmann equivalence}, satisfied by the Galois groups of the two fields contained in a common Galois closure; afterwards, Stuart and Perlis \cite{stuart1995new} were able to show that these relations are equivalent to a third one, called \textit{split equivalence}. Their results are summarized in the following.

\begin{thm}[Perlis, Stuart]\label{thmEquivalence}
Let $K$ and $L$ be number fields. The following are equivalent.
\begin{itemize}
    \item[1)] \underline{Arithmetic equivalence}. For every prime number $p\in\Z$, up to exceptions with Dirichlet density zero, one has $f_K(p)=f_L(p)$.
    \item[2)] \underline{Gassmann equivalence}. If $N$ is a Galois extension of $\Q$ containing both $K$ and $L$, then for every conjugacy class $\mathcal{C}\subset \Gal(N/\Q)$ one has $|\mathcal{C}\cap \Gal(N/K)| = |\mathcal{C}\cap \Gal(N/K)|.$
    \item[3)] \underline{Split equivalence}. For every prime number $p\in\Z$, up to exceptions with Dirichlet density zero, the number of prime ideals of $\Ok$ above $p$ equals the number of prime ideals of $\Ol$ above $p$.
\end{itemize}
\end{thm}
\noindent
This result admits also an analytic counterpart involving the Dedekind zeta function $\zk(s)$ of a number field $K$ \cite{perlis1977equation}.

\begin{thm}[Perlis]\label{thmDedekindZeta}
Let $K$ and $L$ be number fields. Then $K$ and $L$ are arithmetically equivalent if and only if $\zeta_K(s)=\zeta_L(s)$.
\end{thm}
The proof of this analytic equivalence shows that the local factors of the two zeta functions coincide, and from this it follows that the set of exceptions in the definition of arithmetic equivalence of number fields is actually the empty set. Moreover, the study of the residue at $s=1$ of the Dedekind zeta functions and the group-theoretic consequences of Gassmann equivalence yield some constraints on their degrees and equality for some invariants. \cite{perlis1977equation}.
\begin{itemize}
    \item $K$ and $L$ have the same degree, the same number $r_1$ of real embeddings  and the same number $r_2$ of complex embeddings up to complex conjugation.
    \item $K$ and $L$ have the same discriminant.
    \item $K$ and $L$ share the same Galois closure, the same Galois core (i.e. the biggest Galois extension of $\Q$ contained in the field) and the same roots of unity.
    \item If $h_K$ and $R_K$ are the class number and the regulator of the field $K$, then $h_KR_K = h_LR_L$. Notice that in general one may have $h_K\neq h_L$, see \cite{desmitperlis1994zeta}.
    \item If $K/\Q$ is Galois, then $K$ and $L$ are isomorphic.
    \item If $[K:\Q]\leq 6$, then $K$ and $L$ are isomorphic.
\end{itemize}
Since the definition of arithmetic equivalence relies upon the factorization of ideals in a Dedekind domain, it is natural to extend this concept and studying it over global fields, like number fields different from $\Q$ or function fields over finite fields.

In this paper we focus on the study of arithmetic equivalence for global function fields, i.e. finite extensions $K$ of $\Fq(T)$ where $\F_q$ is a finite field and $T$ is a fixed transcendental element over $\Fq$. We say that $K$ and $L$ are \textit{arithmetically equivalent over $\Fq(T)$} if, up to exceptions of zero Dirichlet density, one has $f_K(\mathfrak{p})=f_L(\mathfrak{p})$ for every non-zero prime $\mathfrak{p}$ of $\Fq[T]$, i.e. for every monic irreducible polynomial in $\Fq[T]$. From now on, we will assume that every considered extension is separable and that an algebraic closure of $\Fq(T)$ is fixed.

Cornelissen, Kontogeorgis and Van der Zalm \cite{cornelissen2010arithmetic} studied and characterized arithmetic equivalence over $\Fq(T)$ for \textit{geometric extensions}, i.e. finite and separable extensions $K/\Fq(T)$ such that $\Fq$ is the field of constants of $K$. Their results are summarized in the following.

\begin{thm}[Cornelissen, Kontogeorgis, Van der Zalm]\label{thmCornelissen} 
Let $K$ and $L$ be two geometric extensions of $\Fq(T)$.
\begin{itemize}
    \item[i)] The equivalence among the three relations described in Theorem \ref{thmEquivalence} still holds in this case, where ``number field" is replaced with ``geometric extension", $\Q$ is replaced with $\Fq(T)$ and ``$p\in\Z$" is replaced with ``a non-zero prime ideal $\mathfrak{p}$ of $\Fq[T]$".
    \item[ii)] $K$ and $L$ are arithmetically equivalent with empty set of exceptions if and only if $\zk^{[0]}(s) =\zeta_L^{[0]}(s)$, where $\zk^{[0]}(s)$ is the lifted Goss zeta function defined in \cite[Section 5]{cornelissen2010arithmetic}.
    \item[iii)] If $K$ and $L$ are arithmetically equivalent, then they share the same Weil zeta function. the same degree over $\Fq(T)$, the same genus, the same class number, the same absolute discriminant and the same field of constants. They also have the same Galois closure and the same Galois core over $\Fq(T)$.
    \item[iv)] If $K$ and $L$ are arithmetically equivalent and $K$ is Galois over $\Fq(T)$, then the two extensions are isomorphic.
    \item[v)] If $[K:\Fq(T)]\leq 6$, then $K$ and $L$ are isomorphic over $\Fq(T)$.
\end{itemize}
\end{thm}
The first part of this paper (Section 2 and 3) is devoted to removing the need for the geometric hypothesis and generalize the results of Theorem \ref{thmCornelissen} for finite separable extensions of $\Fq(T)$. In particular, if $K$ is not geometric over $\Fq(T)$ and contains $\F_{q^r}(T)$ with $r\geq 2$, we give a proper definition of lifted Goss zeta function, depending on the chosen rational base field and generalizing the previous one. We will do this by recalling the previous proofs of the equivalence among the several relations for the geometric case, and in doing this we will underline why the equality of Zeta functions needs an empty set of exceptions as hypothesis to be satisfied in the case of global function fields, differently from what we know for number fields.

The second part of this work (Section 4) justifies the need for the previous generalization by presenting an actual example of two non-geometric extensions $K$ and $L$ of $\Fq(T)$ which are arithmetically equivalent and non isomorphic over $\Fq(T)$ but such that they are not equivalent over $\F_{q^2}(T)$. In fact, known examples of non-geometric extensions equivalent over $\Fq(T)$ have been already obtained in \cite{cornelissen2010arithmetic} and by Solomatin \cite{solomatin2016artin} but considering equivalent geometric extensions over $\F_{q^r}(T)$, which therefore are equivalent over $\Fq(T)$. Instead, for every prime number $p$ we present a couple $(K(p), K'(p))$ of equivalent and non-isomorphic extensions of $\F_2(T)$ which contain $\F_4(T)$ but are not equivalent over this bigger field; furthermore, we will show that $K(p)$ is not equivalent to $K(q)$ if $p\neq q$ and $p\geq 5$, so that the examples are distinct from each other. Finally, we will also show why the equality of splitting types have no exceptions for these fields, so that the corresponding lifted Goss zeta functions are equal.

\subsection*{Acnkowledgments} The authors thank Torsten Schoeneberg for useful clarifications about the weak topology on Witt vectors. The first author was supported by the French “Investissements d'Avenir" program, project ISITE-BFC (contract ANR-15-IDEX-0003).

\section{Extenstion of known results to the non-geometric case}
We begin by showing that the equivalence between arithmetic equivalence, Gassmann equivalence and split equivalence still holds if we consider separable extensions of $K/\Fq(T)$ without the geometric hypothesis. From now on, we will denote with $\Ok$ the integral closure of $\Fq[T]$ in $K$.

\begin{thm}
Let $K$ and $L$ be finite separable extensions of $\Fq(T)$. The following are equivalent.
\begin{itemize}
    \item[1)]\underline{Arithmetic equivalence}. For every non-zero prime $\mathfrak{p}$ of $\Fq[T]$ up to exceptions with Dirichlet density zero, one has $f_K(\mathfrak{p})=f_L(\mathfrak{p})$.
    
    \item[2)]\underline{Gassmann equivalence}. If $N$ is a Galois extension of $\Fq(T)$ containing both $K$ and $L$, then for every conjugacy class $\mathcal{C}\subset \Gal(N/\Fq(T))$ one has $|\mathcal{C}\cap \Gal(N/K)| = |\mathcal{C}\cap \Gal(N/K)|.$
    
    \item[3)]\underline{Split equivalence}. For every finite prime $\mathfrak{p}$ of $\Fq[T]$ up to exceptions with Dirichlet density zero, the number of prime ideals of $\Ok$ above $\mathfrak{p}$ equals the number of prime ideals of $\Ol$ above $\mathfrak{p}$.
\end{itemize}
\end{thm}

\begin{proof}
The implication 1)$\Rightarrow 3)$ is trivial. The proof of 3)$\Rightarrow$2) follows the same outline of the one of Theorem \ref{thmCornelissen}.i) as presented in \cite[Proposition 2.1]{cornelissen2010arithmetic}, and this in turn presents an argument which is completely similar to the one of Theorem \ref{thmEquivalence} given by Stuart and Perlis \cite{stuart1995new}. The crucial fact is that both cases depend on the application of Chebotarev's Theorem, either the usual one for number fields or the one for global function fields \cite[Thm 6.3.1]{fried2006field}. We are interested in the latter, the validity of which does not depend on the assumption that the considered extensions are geometric over $\Fq(T)$. Hence, a straightforward application of the steps described in  \cite[Proposition 2.1]{cornelissen2010arithmetic} gives the result. 

In order to prove 2)$\Rightarrow$1), consider a finite prime $\mathfrak{p}$ of $\Fq[T]$ which is unramified in the Galois closure $N$ of both $K$ and $L$: the decomposition group $G_{\mathfrak{B}}$ of a prime $\mathfrak{B}\subset\mathcal{O}_N$ lying over $\mathfrak{p}$ is a cyclic subgroup of $G$. Consider $G\coloneqq \Gal(N/\Fq(T)), H_K\coloneqq \Gal(N/K)$ and $H_L\coloneqq\Gal(N/L)$:  we can write $G$ as disjoint union of double cosets $H_K\tau_i G_{\mathfrak{B}}$ or of double cosets $H_L\psi_j G_{\mathfrak{B}}$ with $\tau_i,\psi_j\in G$. We have \cite{gassmann1926bemerkungen} that $H_K$ and $H_L$ are Gassmann equivalent in $G$ if and only if they have the same coset type for every cyclic subgroup $C$ of $G$, i.e. $G=\coprod_{i=1}^n H_K \tau_i C = \coprod_{i=1}^n H_L \psi_i C$ and $|H_K \tau_i C| = |H_L \psi_i C| = |H_K|\cdot f_i$. Moreover, for every considered prime $\mathfrak{p}$, the numbers $f_i$ appearing in the coset type of $H_K$ with respect to $G_{\mathfrak{B}}$ are exactly the inertia degrees of its splitting type in $K$. Hence $f_K(\mathfrak{p})=f_L(\mathfrak{p})$. Since there is only a finite number of primes which ramifies in $N$, the claim follows.
\end{proof}

\begin{coro}
If $K$ and $L$ are arithmetically equivalent over $\Fq(T)$, then the set of exceptions is finite and is contained in the set of ramified primes in the Galois closure of $K$.
\end{coro}

\begin{coro}
If $K$ and $L$ are arithmetically equivalent and if  their Galois closure $N/\Fq(T)$ is unramified, then the set of exceptions is empty.
\end{coro}

Arithmetic equivalence is linked to the research of Galois groups admitting Gassmann equivalent subgroups; in particular, if $G$ is the Galois group of an extension $N/\Fq(T)$ and $G$ admits two Gassmann equivalent subgroups which are not conjugated in $G$, then Galois Theory implies that $N$ contains two extensions $K/\Fq(T)$ and $L/\Fq(T)$ which are arithmetically equivalent but not isomorphic over $\Fq(T)$.
\begin{prop}\label{propGroupProperties}
Let $K/\Fq(T)$ and $L/\Fq(T)$ be finite separable extensions which are arithmetically equivalent over $\Fq(T)$.
\begin{itemize}
    \item $K$ and $L$ have the same Galois closure over $\Fq(T)$, the same Galois core over $\Fq(T)$ and the same field of constants.
    \item If $K/\Fq(T)$ is Galois, then $K = L$.
    \item If $[K:\Fq(T)]\leq 6$, then $K$ and $L$ are isomorphic over $\Fq(T)$.
\end{itemize}
\end{prop}

\begin{proof}
Let $N/\Fq(T)$ be a Galois extension containing both $K$ and $L$, and consider $G$, $H_K$ and $H_L$ be defined as before. By Galois Theory, the Galois closure of $K$ corresponds to the normal subgroup $\cap_{\sigma\in G} H_K^{\sigma}$ (where $H_K^{\sigma}\coloneqq \sigma^{-1} H_K \sigma$). Since arithmetic equivalence for $K$ and $L$ is the same as Gassmann equivalence for $H_K$ and $H_L$, the conjugacy class $\mathcal{C}_h$ of $h\in \cap_{\sigma\in G} H_K^{\sigma}$ satisfies
$$|\mathcal{C}_h| = |\mathcal{C}_h\cap H_K| = |\mathcal{C}_h\cap H_L|$$
which implies $h\in\cap_{\sigma\in G} H_L^{\sigma}$ and hence the equality of Galois closures by symmetry. A similar argument holds for the Galois core, which corresponds to the fixed field of the subgroup generated by the conjugates $H_K^{\sigma}$ of $H_K$.

Let $\F_{q^r}$ be the field of constants of $K$: then $K/\Fq(T)$ contains a cyclic subextension $\F_{q^r}(T)/\Fq(T)$ which is contained in the Galois core of $K/\Fq(T)$ by definition, hence is a subextension of $L/\Fq(T)$ by the previous result. Applying the symmetric procedure to $L/\Fq(T)$ we obtain that the two extensions have the same field of constants (notice that the constants are indeed the roots of unity of these extensions). If $K/\Fq(T)$ is Galois, then $H_K=G$ and must contain $H_L$ by the previous results. But Gassmann equivalent subgroups have the same caridinality, hence $H_K=H_L$ and the claim follows.

Finally, if $[K:\Fq(T)]=n\leq 6$, then the Galois closure of $K/\Fq(T)$ has Galois group equal to a transitive subgroup of $S_n$. Perlis \cite[Theorem 3]{perlis1977equation} proved that there are no transitive subgroups of $S_n$ with $n\leq 6$ which admit Gassmann equivalent subgroups which are not conjugated. 
\end{proof}

\begin{prop}\label{propWeil}
If $K/\Fq(T)$ and $L/\Fq(T)$ are arithmetically equivalent, then they have the same Weil zeta function.
\end{prop}

\begin{proof}
This is proved in \cite[Proposition 3.2]{cornelissen2010arithmetic} for arithmetically equivalent fields over a fixed global field, but we recall the proof for sake of completeness. Let $N/\Fq(T)$ be the Galois closure of $K/\Fq(T)$ with group $G$ and subgroup $H_K$ corresponding to $K$. If $\F_{q^r}$ is the field of constants of $K$, the Weil zeta function of $K$ is defined as $\zk(s)\coloneqq \sum_{I}\Nm(I)^{-s}$ where $s\in\C$, the sum is made over divisors of $K$ and the absolute norm is made over the field of constants $\F_{q^r}$ (and not $\Fq$).  Let $G_r\coloneqq\Gal(N/\F_{q^r}(T))$: this group still contains $H_K$ as a subgroup, and we have $\text{Spec }\mathcal{O}_N / H_K = \text{Spec }\Ok$. Notice that the result of this quotient does not depend on whether $H_K$ is seen as a subgroup of $G$ or $G_r$.

If $1_{H_K}^G$ is the representation of $G$ induced by the trivial representation of $H_K$, then we have that $K$ and $L$ are arithmetically equivalent if and only if $1_{H_K}^G \simeq 1_{H_L}^G$ \cite[Chapter 3, Lemma 2]{stuart1995new}. Using Serre's formalism for zeta functions \cite{serre1965zeta}, we have from the work by Kani and Rosen \cite{kani1994idempotent} that
$$
L(\text{Spec }\mathcal{O}_N, 1_{H_K}^{G_r},s ) =\zk(s) = L(\text{Spec }\mathcal{O}_N, 1_{H_K}^{G},s ).
$$
Now, since the representations $1_{H_K}^G$ and $1_{H_L}^G$ are isomorphic, we obtain\\
$ L(\text{Spec }\mathcal{O}_N, 1_{H_K}^{G},s ) =  L(\text{Spec }\mathcal{O}_N, 1_{H_L}^{G},s )$ and from this the desired equality of Weil zeta functions.
\end{proof}

\begin{rmk}
Notice that in the proof above we obtained $ L(\text{Spec }\mathcal{O}_N, 1_{H_K}^{G_r},s )= L(\text{Spec }\mathcal{O}_N, 1_{H_L}^{G_r},s )$ but this does not imply that the representations $1_{H_K}^{G_r}$ and $1_{H_L}^{G_r}$ are isomorphic, i.e. that the arithmetic equivalence of $K$ and $L$ over $\Fq(T)$ implies the arithmetic equivalence of the two fields over $\F_{q^r}(T)$. We shall see an explicit counterexample in later sections.
\end{rmk}
Now we list some common invariants that arithmetically equivalent fields over $\Fq(T)$ share. 

\begin{coro}
Let $K/\Fq(T)$ and $L/\Fq(T)$ be arithmetically equivalent extensions. Then they share the same genus, the same class number, the same degree of their different ideals over $\Fq(T)$, and their adele rings are isomorphic.
\end{coro}

\begin{proof}
The fields have the same Weil zeta function, which has the form
$$\zk(s)=\frac{L_{2g}(q^{-rs})}{(1-q^{-rs})(1-q^{r(1-s)})}$$
where $q^r$ is the cardinality of the field of constants of $K$, $g$ is the genus of $K$ (notice that the value of the genus does not depend on whether $K$ is seen as an extension $K/\Fq(T)$ or $K/\F_{q^r}(T)$ since $\F_{q^r}(T)/\Fq(T)$ is an unramified extension) and $L_{2g}(x)$ is a complex polynomial of degree $2g$ in the variable $x$.

The arithmetically equivalent fields must have equal field of constants and equal Weil zeta function by the previous result, thus they share the polynomial $L_{2g}$ and they have the same genus. The class number of $K$ is equal to $L_{2g}(1)$ (see \cite[Theorem 5.1.15]{stichtenoth2009algebraic}) and so it is equal to the class number of $L$. The fields have also same degree of their different ideals thanks to the Hurwitz Genus Formula
$$2g-2 = -2[K:\Fq(T)] + \deg \text{Diff}(K/\Fq(T)).$$
Finally, the fields $K$ and $L$ have isomorphic adele rings since they have the same Weil zeta function: this is proved in \cite[Theorem 1]{turner1978adele}.
\end{proof}

\begin{rmk}
Two finite and separable extensions of $\Fq(T)$ have isomorphic adele rings if and only if they share the same Weil zeta function \cite[Theorem 2]{turner1978adele}; in particular, thanks to Proposition \ref{propWeil}, if the two fields are arithmetically equivalent then they have isomorphic adele rings. The converse however does not hold: a typical counterexample, as explained in \cite[Section 2]{cornelissen2010arithmetic}, is given by $\F_3(\sqrt{T})$ and $\F_3(\sqrt{T+1})$ as quadratic extensions of $\F_3(T)$, which have the same Weil zeta function, since they are isomorphic as absolute fields, but cannot be equivalent over $\F_3(T)$ because otherwise they would be isomorphic as $\F_3(T)$-algebras thanks to Proposition \ref{propGroupProperties}.
If we consider number fields, the relation between arithmetic equivalence and adele rings behaves differently: in this case, isomorphic adele rings give arithmetic equivalence (see \cite[Lemma 1]{komatsu1984adele}, which is based upon \cite[Lemma 7]{iwasawa1953rings}) but the converse does not hold, a counterexample given by the equivalent and non isomorphic number fields defined by the polynomials $X^8-97$ and $X^8-16\cdot 97$ \cite[pp.351-352]{perlis1977equation}.
\end{rmk}

\section{Lifted Goss Zeta function}
In this section we give the definition of lifted Goss Zeta function for a separable extension $K/\F_{q^r}(T)$: this recovers the definition presented in \cite[Section 5]{cornelissen2010arithmetic} but we stress that for this more general case it is very important to define explicitly the base field of the extension, since we are considering also non-geometric extensions.

Let $F\coloneqq \Fq(T)$ with characteristic $p$. The place at infinity yields a discrete valuation $v_{\infty}$ on $F$, with uniformizer $T^{-1}$, and an absolute value $||\cdot||_{\infty}$ therefore. Denote $F_{\infty}$ to be the completion of $F$ with respect to this absolute value and $\C_{F}$ to be the completion of a fixed algebraic closure of $F_{\infty}$ with respect to the unique extension of $||\cdot||_{\infty}$. We remark that $\C_{\Fp(T)}=\C_F$, and we denote its uniformizer with $\pi$; every element $\alpha\in\C_F$ can be written uniquely \cite[Section 7.2]{goss2012basic} as $\alpha = \pi^{v_{\infty}(\alpha)}\zeta_{\alpha}\langle\alpha\rangle$, where $\zeta_{\alpha}$ is a root of unity in $\C_F$ and $\langle\alpha\rangle = 1 + m_{\alpha}$ with $\Vert m_{\alpha}\Vert_{\infty}<1$ (in particular, $\langle\alpha\rangle$ is a 1-unit of $\C_F$ and $\Vert\langle\alpha\rangle\Vert_{\infty}=1$). If $\alpha=f\in\Fq[T]$ is a monic polynomial, then $\zeta_f =1$.

Let $K/\Fq(T)$ be a finite extension and $\mathfrak{P}\subset\Ok$ a finite prime, and let $f\in\Fq[T]$ be a monic irreducible polynomial which generates the prime ideal $\mathfrak{P}\cap\Fq[T]$, so that the \textit{inertia degree} of $\mathfrak{P}$ over $\Fq(T)$ is equal to $\left[\frac{\Ok}{\mathfrak{P}}: \frac{\Fq[T]}{(f)}\right]$. The \textit{norm} of $\mathfrak{P}$ over $\Fq(T)$ is defined as 
$$N_{q}^K(\mathfrak{P})\coloneqq f^{\left[\frac{\Ok}{\mathfrak{P}}: \frac{\Fq[T]}{(f)}\right]}$$
and it defines a function from the ideals of $\Ok$ to $\Fq[T]$ if we extend it by multiplicativity. If $K=\F_{q^r}(T)$, we denote the norm as $N^{q^r}_q$, and if $K$ is a finite separable extension of $\F_{q^r}(T)$ we have
\begin{equation*}
    N_q^K(\mathfrak{P}) = N_q^{q^r}(N_{q^r}^K(\mathfrak{P})).
\end{equation*}

Before giving the definition of lifted Goss zeta function, we need an intermediate step. Let $K/\Fq(T)$ be a finite separable extension of global function fields, and define $S_{\infty}\coloneqq \C_{\Fq(T)}^*\times \Z_p$, where $p$ is the characteristic of the fields and $\Z_p$ are the $p$-adic integers. The definition of $S_{\infty}$ only depends on the characteristic of $\Fq(T)$ and $S_{\infty}$ becomes a topological group with the product topology. Given $c>0$, we call a subset of $S_{\infty}$ of the form $\{s=(x,y)\in S_{\infty}\colon \Vert x\Vert_{\infty}>c\}$ a \textit{half-plane}.

If $s=(x,y)\in S_{\infty}$ and $f\in\Fq[T]$ is a monic polynomial, define the exponential
\begin{equation*}
    f^s \coloneqq x^{-v_{\infty}(f)}\langle f\rangle^y
\end{equation*}
where we recall that, if $a_0+a_1p+\cdots+a_np^n+\cdots$ is the $p$-adic expansion of $y$ with $0\leq a_i\leq p-1$ for every $i$, the term $\langle f\rangle^y$ is naturally given by the convergent product
$$
\langle f\rangle^y\coloneqq \prod_{i=0}^n\lp 1+m_f^{p^i}\rp^{a_i}.
$$
Notice that $-v_{\infty}(f)=\deg f$ by definition.
\begin{lem}
Let $f,g\in\Fq[T]$ be monic polynomials and let $s=(x,y), s'=(x',y')\in S_{\infty}$.
\begin{itemize}
    \item We have $f^{s+s'}=f^sf^{s'}$ and $(fg)^s = f^sg^s$.
    \item If $j\in\Z$ and $s_j\coloneqq (\pi^{-j},j)$, then $f^{s_j} = f^j$ (where the second exponential is thought in the usual way).
    \item We have $f^{-s}=(f^s)^{-1}$.
\end{itemize}
\end{lem}

\begin{proof}
These results are described in the propositions 8.1.3 and 8.1.4 of \cite{goss2012basic}.
\end{proof}
The previous lemma guarantees that integer exponentiation is always possible, since every $j\in\Z$ is now identified with $s_j\in S_{\infty}$ (and in this way $\Z$ is seen as a subgroup of $S_{\infty}$). The \textit{Goss zeta function} is then defined as
\begin{equation*}
    \zeta_{K/\Fq(T)}^{[p]}(s)\coloneqq \sum_{\mathfrak{I}}\frac{1}{N^K_q(\mathfrak{I})^s}
\end{equation*}
where the sum runs over the non-zero ideals of $\Ok$. This series can be rewritten also as
\begin{equation}\label{GossZeta}
    \zeta_{K/\Fq(T)}^{[p]}(s)\coloneqq \sum_{f}\frac{A_K(f)}{f^s}
\end{equation}
where the sum runs over the monic polynomials $f\in\Fq[T]$ and $A_K(f)\in \Fp$ is the number mod $p$ of ideals $\mathfrak{I}\subset\Ok$ with norm $f$. In the following lemma we show that this series admits a non-trivial half-plane of convergence in $S_{\infty}$
\begin{lem}
The Goss zeta function converges for every $s=(x,y)\in S_{\infty}$ such that $\Vert x\Vert_{\infty}>1.$
\end{lem}

\begin{proof}
Since $A_K(f)\in\Fp$, we have either $\Vert A_K(f)\Vert_{\infty}=0$ or $\Vert A_K(f)\Vert_{\infty}=1$. The norm of the general term of the Goss zeta function is then \begin{align*}
    \left\Vert\frac{A_K(f)}{f^s}\right\Vert_{\infty}\leq \frac{1}{\Vert x\Vert_{\infty}^{\deg f}\Vert \langle f\rangle^y\Vert_{\infty}} = 
    \frac{1}{\Vert x\Vert_{\infty}^{\deg f}}
\end{align*}
since $\langle f\rangle$ is a unit in the ring of integers of $\C_F$. This upper bound goes to 0 as $\deg f$ goes to infinity, and this implies the convergence of the series.
\end{proof}
We remark that knowing the Goss zeta function is equivalent to know all the coefficients $A_K(f)\in\Fp\subset F_{\infty}$ thanks to the following lemma on general Dirichlet series \cite[Theorem 8.7.1]{goss2012basic}.

\begin{lem}\label{lemmaDirichlet}
Let $L(s)\coloneqq \sum_{f}\frac{c(f)}{f^s}$ be a Dirichlet series running over the monic polynomials $f\in\Fq[T]$ and with $c(f)\in F_{\infty}$. Assume $L(s)$ converges over a non-trivial half-plane of $S_{\infty}$: then $L(s)$ uniquely determines the coefficients $c(f)$.
\end{lem}

However, the Goss zeta function does not provide an analytic counterpart to arithmetic equivalence, i.e. the result of Theorem \ref{thmDedekindZeta} cannot be recovered in the global function field case if we replace the Dedekind zeta function with the Goss zeta function. This is due to the coefficients $A(f)$, which are not non-negative integers but integers modulo $p$, and so they only provide partial information on the splitting types in the considered global function field extensions. In \cite{cornelissen2010arithmetic}, the authors obtained the analogue of Theorem \ref{thmDedekindZeta} for the Goss zeta function only if the degrees of the equivalent extensions was strictly less than $p$, and they gave an explicit counterexample when the degree is exactly $p$.

This problem suggests that one should look for a lifting in characteristic 0 of the Goss zeta function: this is exactly what has been done in \cite{cornelissen2010arithmetic} by means of the Witt vectors. In the next lines we recall the definition of the lifted zeta function, without the assumption that the considered extensions are geometric, and the needed properties of Witt vectors for this function to be defined: references for these technical facts can be found in \cite[Chapter 2]{serre2013local}, \cite[Appendix A]{fontaine2008theory} and in \cite[Chapter 1]{schneider2017galois}.

Let $W$ be the ring of Witt vectors of $\C_F$: this is a ring with set $\C_F^{\N}$ and sum and product given by relations between \textit{Witt polynomials}. Since $\C_F$ is a perfect field of characteristic $p$, then $W$ is a complete ring with respect to the $p$-adic topology, has uniformizer $p$ and residue field $\C_F$. If $pr: W\to\C_F$ is the surjective ring morphism given by the reduction modulo $p$,  the \textit{Teichm\"uller character} is the unique homomorphism of multiplicative groups $\chi: \C_F^*\to W^*$ such that
$$
\begin{tikzcd}
&W^*\arrow{r}{pr} &\C_F^*\\
&\C_F^*\arrow{u}{\chi}\arrow{ru}{id}. &
\end{tikzcd}
$$
which means that $\chi(a)\equiv a\bmod p$ for every $a\in\C_F^*$. We extend $\chi$ so that $\chi(0)=0$: then every Witt vector $x=(x_0,x_1,\ldots,)\in W$, with $x_i\in\C_F$ for every $i$, can be written uniquely as $x=\sum_{i=0}^{\infty}\chi\lp x_i^{1/p^i}\rp p^i$.
In particular, for every $a\in\C_F$ we have
$$
\begin{matrix}
    \chi(a)&=&(a,0,0,\ldots),\\
    p^n\chi(a)&=&(\underbrace{0,\ldots,0}_{n-1},a^{p^n},0,\ldots).
\end{matrix}
$$
If $a\in\C_F$ and $x=(x_n)_{n\in\N}$, one also has $\chi(a)x = (a^{p^n} x_n)_{n\in\N}$. Finally, if $n\in\N$, then $n=\sum_{i=0}^{\infty}\chi(n_i)p^i$ with $n_i\in\Fp$ for every $i$.



Now, we consider a weaker topology on $W$ than the $p$-adic one. Let $\mathcal{O}_{\C_F}\coloneqq \{x\in\C_F\colon \Vert x\Vert_{\infty}\leq 1\}$: for every open ideal $I\subset \mathcal{O}_{\C_F}$ and for every $m\in\N$, define 
$$V_{I,m}\coloneqq \{(x_n)_{n\in\N}\in W\colon x_0,x_1,\ldots,x_m\in I\}.$$
These sets define a fundamental system of neighbourhoods of 0 and hence a topology in $W$. In particular, the sequence of Witt vectors $w_k=((x_{k,n})_{n\in\N})$ converges to 0 under this topology if for every $m, r\in\N$ there exists $k_0\in\N$ such that $x_{k,0},\ldots,x_{k,m}\in \pi^r\mathcal{O}_{\C_F}$ for every $k\geq k_0$.

\begin{lem}\label{lemConvergence}
The ring $W$ is complete with respect to the weak topology, and any Cauchy sequence with respect to the $p$-adic topology is still Cauchy in the weak topology. Moreover, a series of Witt vectors $\sum x_k$ converges with respect to this topology if and only if $x_k$ converges to 0.
\end{lem}

\begin{proof}
This is explained in \cite[Chapter 1.5]{schneider2017galois} and \cite[Chapter III.5]{bourbaki2007topologie}.
\end{proof}

We are now ready the define the zeta function we want. If $K/\Fq(T)$ is a finite separable extension, the associated \textit{lifted Goss zeta function} is defined as
\begin{align*}
    \zeta_{K/\Fq(T)}^{[0]}(s)\coloneqq \sum_{\substack{\mathfrak{I}\subset\Ok\\\mathfrak{I}\text{ non-trivial}}}\chi(\Nm_q^K(\mathfrak{I})^{-s}) = \sum_{\substack{f\in\Fq[T]\\f\text{ monic }}}B_K(f)\chi(f^{-s})
\end{align*}
where $B_K(f)$ is the (characteristic 0) number of ideals of $\Ok$ with norm equal to $f$. 
\begin{prop}
The lifted Goss zeta function converges for $s=(x,y)\in S_{\infty}$ with $\Vert x\Vert_{\infty} >1.$
\end{prop}

\begin{proof}
Since $B_K(f)$ is a non-negative integer, its expansion in $W$ has the form\\ $(c_0(f),c_1(f),\ldots,c_n(f),\ldots)$ with $c_i(f)\in\F_p$ for every $i$. In particular, $c_0(f)\equiv B_K(f)\bmod p$, i.e. $c_0(f)=A_K(f)$. 
The lifted Goss zeta function can be rewritten as
\begin{align*}
    \sum_{\substack{f\in\Fq[T]\\f\text{ monic }}}B_K(f)\chi(f^{-s}) &= 
     \sum_{\substack{f\in\Fq[T]\\f\text{ monic }}}(c_n(f))_{n\in\N}\cdot (f^{-s},0,0,\ldots)\\
     &=\sum_{i=0}^{\infty}\lp\frac{c_n(f_i)}{f_i^{p^n s}}\rp_{n\in\N}
\end{align*}
where the monic polynomials of $\Fq[T]$ have been ordered increasingly in degree and coefficients ($f_0=1, f_1=T, f_2=T+1$ and so on). Let $s=(x,y)\in S_{\infty}$ with $\Vert x\Vert_{\infty}>1$: the $n$-th component of the $i$-th term of the series has absolute value in $\C_F$ equal to

$$
\left\Vert \frac{c_n(f)}{f_i^{p^ns}}\right\Vert_{\infty} \leq \Vert x\Vert_{\infty}^{-\deg f_i}
$$
and therefore every component has uniformly small absolute value as the degree of $f_i$ increases, i.e. the general term of the series converges to 0, hence the series converge by Lemma \ref{lemConvergence}. 
\end{proof}

Notice that the previous result implies $\zeta_{K/\Fq(T)}^{[0]}(s) \equiv \zeta_{K/\Fq(T)}^{[p]}(s) \bmod p.$ In particular, every zero of the lifted Goss zeta function is a zero for the characteristic $p$ Goss function.

\begin{coro}\label{corocoefficients}
Let $K$ and $L$ be finite separable extensions of $\Fq(T)$. Then $\zeta_{K/\Fq(T)}^{[0]}(s) = \zeta_{L/\Fq(T)}^{[0]}(s)$ if and only if $B_K(f)=B_L(f)$ for every monic polynomial $f\in\Fq[T]$.
\end{coro}

\begin{proof}
We show that equality of zeta functions determines equality of the coefficients (the converse implication is trivial). Let $s=(x,y)\in S_{\infty}$ with $\Vert x\Vert_{\infty}>1$ and let us write
$$ 
\sum_{i=0}^{\infty}\frac{B_K(f_i)}{\chi(f_i^s)} = \sum_{i=0}^{\infty}\lp\frac{c_n(f_i)}{f_i^{p^n s}}\rp_{n\in\N}=
\sum_{i=0}^{\infty}\lp\frac{d_n(f_i)}{f_i^{p^n s}}\rp_{n\in\N}=
\sum_{i=0}^{\infty}\frac{B_L(f_i)}{\chi(f_i^s)}.
$$
By the rules of Witt vectors, the 0-component of the zeta function is the only one which is equal to the sum of the 0-components of the terms in the series, and thus we obtain
$$
\sum_{i=0}^{\infty}\frac{c_0(f_i)}{f_i^s} = \sum_{i=0}^{\infty}\frac{d_0(f_i)}{f_i^s}
$$
which yields $c_0(f_i)=d_0(f_i)$ for every $i$ from Lemma \ref{lemmaDirichlet}.

For what concerns the addition in the remaining components, the rules of Witt vectors are such that the $n$-component is given by
$$\sum_{i=0}^{\infty}\frac{c_n(f_i)}{f_i^{p^ns}} + F_n((f_i^s,c_0(f_i),\ldots,c_{n-1}(f_i))_{i\in\N})$$
where $F_n$ is a precise series which converges under our hypotheses and only depends on the polynomials $f_i$ and on the coefficients $c_0(f_i),\ldots,c_{n-1}(f_i)$. An induction argument then shows that the equality of the $n$-components gives
$$
\sum_{i=0}^{\infty}\frac{c_n(f_i)}{f_i^{p^n s}} = \sum_{i=0}^{\infty}\frac{d_n(f_i)}{f_i^{p^n s}}
$$
and thus $c_n(f_i)=d_n(f_i)$ for every $i$. Applying this for every $n\in\N$, we finally obtain $B_K(f_i)=B_L(f_i)$ for every $i$.
\end{proof}

\begin{thm}\label{ThmAnalyticEquiv}
Let $K$ and $L$ be finite separable extensions of $\Fq(T)$. Then $K$ and $L$ are arithmetically equivalent over $\Fq(T)$ with no exceptions if and only if $\zeta_{K/\Fq(T)}^{[0]} = \zeta_{L/\Fq(T)}^{[0]}$.
\end{thm}

\begin{proof}
Let $s\in\N$, $s\geq 1$ and express the lifted Goss zeta function of $K/\Fq(T)$ at $s$ as
$$\sum_{\substack{f\in\Fq[T]\\f\text{ monic }}}\frac{B_K(f)}{\chi(f)^s}.$$
Notice that if $f$ and $g$ are coprime irreducible polynomials in $\Fq[T]$ then 
$$B_K(f^m g^n) = B_K(f^m)+B_K(g^n)$$
so that $B_K(\cdot)$ is an additive (but not completely additive) function over the prime ideals of $\Fq[T]$.

Now, for every $\mathfrak{p}\in\Fq[T]$ monic and irreducible and for every $m\in\N$, let $C_K(\mathfrak{p}^m)$ be the number of prime ideals of $\Ok$ having norm equal to $\mathfrak{p}^m$. Unique factorization as product of prime ideals implies the following relation (see \cite[Theorem 2.1]{klingen1998arithmetical}):
\begin{align*}
    B_K(\mathfrak{p}^m) = \sum_{\substack{a_1m_1+\cdots+a_r m_r=m\\1\leq a_1<\cdots<a_r\\1\leq m_1,\ldots,m_r}}\prod_{i=1}^r\binom{C_K(\mathfrak{p}^{m_i})+a_i-1}{a_i}
\end{align*}
and the right hand side of this expression can be rewritten as $C_K(\mathfrak{p}^m)$ plus a function of the numbers $C_K(\mathfrak{p}^{m'})$ with $m'<m$. This shows not only that knowing all numbers $C_K(\mathfrak{p}^{m'})$ determines all numbers $B_K(\mathfrak{p}^m)$, but also that the converse holds, since $B_K(\mathfrak{p})=C_K(\mathfrak{p})$ and the values for successive powers are obtained with a recursive procedure.

But now, knowledge of the numbers $B_K(\mathfrak{p}^m)$ for every irreducible monic polynomial $\mathfrak{p}\in\Fq[T]$ gives the lifted Goss zeta function, while knowledge of the numbers $C_K(\mathfrak{p}^m)$ gives the splitting types of every finite prime of $\Fq(T)$ in $K$, and the claim follows from Corollary \ref{corocoefficients}.
\end{proof}

\begin{rmk}
The lifted Goss zeta function can be rewritten as an Euler product
\begin{align*}
    \prod_{\mathfrak{P}\subset\Ok}\lp 1-\chi(\Nm_q^K(\mathfrak{P})^{-s})\rp^{-1}
\end{align*}
which converges over the half-plane  $\{s=(x,y)\in S_{\infty}\colon \Vert x\Vert_{\infty} >1\}.$
\end{rmk}

\begin{rmk}
It is important to underline that, in the course of the proof of Theorem \ref{ThmAnalyticEquiv}, equivalence has been established among equal lifted Goss zeta functions and arithmetic equivalence \textit{with no exceptions}: in fact, an arithmetic equivalence with a non-zero number of exceptions would imply that the two zeta functions coincide up to finitely many factors of their Euler products. If we divided $\zeta_{K/\Fq(T)}^{[0]}$ with $\zeta_{L/\Fq(T)}^{[0]}$ we would obtain then an expression of the form
$$\frac{\prod_{i=1}^a\lp 1-\chi(N_q^K(\mathfrak{P}_i)^{-s})\rp}{\prod_{j=1}^b(1-\chi(N_q^K(\mathfrak{Q}_j)^{-s}))}$$
but there is no obvious reason why this quotient should be equal to 1, differently from what we know for number fields: the correspondent quotient in the number fields case is equal to 1 thanks to reasons which are both analytic (the variable $s$ belongs to $\C$ and the zeta functions are known to satisfy a functional equation) and arithmetic (for every prime appearing in the product, its characteristic explicitly appears thanks to the norm). See \cite[Lemma 2]{perlis1977equation}.
\end{rmk}

\begin{coro}
If $K$ and $L$ are arithmetically equivalent and their Galois closure is unramified at finite places, then $\zeta_{K/\Fq(T)}^{[0]}(s) = \zeta_{L/\Fq(T)}^{[0]}(s)$.
\end{coro}

\section{Arithmetically equivalent siblings over $\F_2(T)$ but not over $\F_4(T)$}
The algebraic definition of arithmetic equivalence via splitting types immediately shows that two extensions $K/\F_{q^r}(T)$ and $L/\F_{q^r}(T)$ equivalent over $\F_{q^r}(T)$ are equivalent also over $\Fq(T)$. But does the converse hold? In particular, let $K$ and $L$ be two non-geometric and equivalent extensions over $\Fq(T)$: they have the same fields of constants $\F_{q^r}$ by Proposition \ref{propGroupProperties}. Does this imply that $K$ and $L$ are arithmetically equivalent over $\F_{q^r}(T)$? 

The answer is no, and this is immediate to see in the following cases: consider a quadratic extension $K$ of $\F_{q^r}(T)$ such that $K/\Fq(T)$ is separable but not Galois: by Galois theory, this implies that there exists an isomorphic extension $L/\Fq(T)$ such that $K\neq L$: in particular, $L/\F_{q^r}(T)$ is a quadratic extension which cannot be equivalent to $K/\F_{q^r}(T)$, since otherwise $K=L$ by Proposition \ref{propGroupProperties}. These extensions do exist: as we will see in further lines, it is not difficult to construct a non-Galois degree 4 extension of $\F_2(T)$ which contains $\F_4(T)$.  However, these examples are of limited interest, since we consider fields which are isomorphic (and hence equivalent) and for which the equality of zeta functions is trivially guaranteed by the isomorphism. What about equivalent but not isomorphic fields over $\Fq(T)$? 

Two fields $K$ and $L$ arithmetically equivalent over $\Fq(T)$ but not isomorphic as extensions of $\Fq(T)$ will be called \textit{arithmetically equivalent siblings}. The goal of this section is to look for non-geometric arithmetically equivalent siblings over $\Fq(T)$ with field of constants $\F_{q^r}$ such that they are not equivalent over $\F_{q^r}(T)$; this research is equivalent to solve an Inverse Galois Problem for a group $G$ which contains two Gassmann equivalent but not conjugated subgroups and such that the fields corresponding to these subgroups contain the cyclic extension $\F_{q^r}(T)/\Fq(T)$.

\begin{rmk}
From now on, the \textit{Galois group of an extension} will be the Galois group of its Galois closure. Following the Butler-McKay notation \cite{butler1983transitive}, the Galois group of a field of degree $n$ will be labeled as $n$T$i$, meaning that it is isomorphic to the $i$-th transitive subgroup of $S_n$ (for complete lists of transitive subgroups of $S_n$ with small $n$ see the database \cite{klunersMalle}).
\end{rmk}

\begin{prop}
Let $K$ and $L$ be non-geometric arithmetically equivalent siblings over $\Fq(T)$ with field of constants $\F_{q^r}$ and degree $d$ at most 8. Let $G$ be their Galois group: then $d=8, r=2$ and $G\simeq C_8\rtimes V_4$, where $C_8$ is the cyclic group with 8 elements and $V_4$ is the Klein group.
\end{prop}

\begin{proof}
If $d\leq 6$, then $K$ and $L$ are isomorphic over $\Fq(T)$ by Proposition \ref{propGroupProperties}, so that they are not siblings. If $d=7$, then $K=\F_{q^7}(T)=L$ since the extensions are non-geometric. If $d=8$, it is known that there are only two possibilities for $G$ \cite[Theorem 3]{bosma2002arithmetically}, which are either the group 8T15 (which is exactly $C_8\rtimes V_4$) or the group 8T23. However, the second case is not possible since an extension of degree 8 with group 8T23 does not contain non-trivial cyclic subextensions (see the properties of fields with this Galois group at the corresponding page in the LMFDB database \cite{lmfdb}). Instead, the maximal cyclic subextension in a field of degree 8 with group 8T15 is 2, so that this is the only available case.
\end{proof}
\begin{rmk}
If $K$ and $L$ are non-geometric arithmetically equivalent siblings over $\Fq(T)$ with group 8T15, then they cannot be equivalent over $\F_{q^2}(T)$ since the two relative extensions have degree 4 and they would be forced to be isomorphic by Proposition \ref{propGroupProperties}, forcing the isomorphism between $K$ and $L$ over $\Fq(T)$.
\end{rmk}
Extensions with group 8T15 in the number field case were explicitly constructed by Perlis \cite{perlis1977equation}: in fact, he proved that the number field $K=\Q(\alpha)$ generated by a chosen root $\alpha$ of $X^8-a$ (with $a\in\Z$ and $a\neq n^2$, $a\neq \pm 2n^2$ for $n\in\Z$) has Galois group 8T15. Moreover, its Galois closure is $N=\Q(\alpha,\zeta_8)=KL$, where $\zeta_8$ is a primitive 8-th root of unity and $L=\Q(\zeta_8)=\Q(i,\sqrt{2})$ and one has $\Gal(N/L)\simeq C_8$ and $\Gal(N/K)\simeq V_4$: finally, the field of degree 8 generated by $\sqrt{2}\alpha$, corresponding to the polynomial $X^8-16a$, is contained in $N$ and is an arithmetically equivalent sibling of $K$, since Perlis proved that $\Gal(N/\Q(\sqrt{2}\alpha))$ is Gassmann equivalent but not conjugated to $\Gal(N/K)$. This was obtained thanks to the following lemma \cite{perlis1977equation}.

\begin{lem}\label{lemmaCohomologic}
Let $G= A\rtimes H$ be a finite group with $A$ abelian and $H$ containing a $p$-sylow which is not cyclic for some prime $p$. Then $G$ contains two Gassmann equivalent subgroups which are not conjugated. More explicitly, if $\sigma : H\to G$ is a section of the semidirect product and if $\chi:H\to A$ is non-trivial in the cohomology group $H^1(H,A)$ but is trivial in every $H^1(\langle h\rangle,A)$ for every cyclic subgroup $\langle h\rangle\subset H$, then $(\chi\cdot\sigma): H\to G$ is a section and $\sigma(H)$ and $(\chi\cdot\sigma)(H)$ are Gassmann equivalent but not conjugated.
\end{lem}

Since the properties of the group 8T15 are unchanged whether it is realized as Galois group of a number field or of a global function field, we take inspiration from Perlis' work in the research of a similar example for global function fields. The choice of the polynomial is obviously harder, since this time the auxiliary biquadratic extension cannot be a cyclotomic extension (i.e. a constant field extension) because every such extension is cyclic in the case of global function fields, so that Lemma \ref{lemmaCohomologic} cannot be applied. However, the example is good enough to inspire a solution for the problem in characteristic 2.

\begin{thm}
There exist two non-geometric arithmetically equivalent siblings $K$ and $K'$ over $\F_2(T)$ with group $C_8\rtimes V_4$ such that they are not equivalent over $\F_4(T)$. They are defined by the polynomials
\begin{align}\label{polK}
    X^8 &+ TX^6 + TX^5 + (T^5 + T^4 + T^3 + T^2 + T + 1)X^4 + TX^3 \nonumber\\
    &+ (T^7 + T^5 + T^3 + T^2)X^2 + (T^7 + T^4)X + (T^{10} + T^8 + T^6)
\end{align}
and
\begin{align}\label{polL}
    X^8 &+ TX^6 + TX^5 + (T^5 + T^3 + T^2 + T + 1)X^4 + TX^3\nonumber\\
    &+(T^5 + T^4 + T^3 + T^2)X^2 + T^4X + (T^{12} + T^{11} + T^9 + T^7 + T^6).
\end{align}
\end{thm}

\begin{proof}
Consider the following tower of quadratic extensions over $\F_2(T)$
$$
\begin{tikzcd}
&K& X^2+X+(T^2+\zeta T)\alpha& (\beta)\\
&F\arrow{u}& X^2+X+\zeta T& (\alpha)\\
&\F_{4}(T)\arrow{u}&X^2+X+1 &(\zeta)\\
&\F_2(T).\arrow{u}&
\end{tikzcd}
$$
where for every extension we indicate a defining polynomial and a fixed root of the polynomial. In particular, $\zeta$ is a fixed third root of unity in $\F_4$ and $F=\F_4(T)(\alpha)$, $K=F(\beta)$.
\\\\
\textit{1) The extensions above are really quadratic extensions.} This is equivalent to show that every polynomial above is irreducible over the corresponding base field: since they are Artin-Schreier polynomials of degree 2 for characteristic 2 fields, they are irreducible if and only if they have no roots, i.e. if and only if their degree 0 terms are not of the form $a^2+a$ with $a$ belonging to the ring of integers of the base field. This is immediate to see for the elements $1\in\F_2[T]$ and $\zeta T\in\F_4[T]$. For the remaining polynomial, notice that $\Of=\F_4[\alpha]$, since $\alpha$ is a transcendent element over $\F_4$ and $T=(\alpha^2+\alpha)/\zeta$, and so the polynomial ring $\F_4[\alpha]$ is the ring of integers of $F$ since it is integrally closed in its field of fractions $F$: but then 
$$(T^2+\zeta T)\alpha = \lp\frac{\alpha^4+\zeta\alpha^2+(\zeta+1)\alpha}{\zeta+1}\rp\alpha$$
and this element has not the form $a^2+a$ with $a\in\F_4[\alpha]$, since it has odd degree in $\alpha$.\\\\
\textit{2) The tower above is the subfield lattice of $K$.} Let us consider first the quartic extension $F/\F_2(T)$: its defining polynomial over $\F_2(T)$ is obtained multiplying its defining polynomial over $\F_4(T)$ with its conjugated over $\F_2(T)$, and so it is equal to
\begin{align*}
    (X^2+X+\zeta T)(X^2+X+(\zeta+1)T) = X^4+(T+1)X^2+TX+T^2.
\end{align*}
One can show then, by employing Galois Theory for quartic polynomials \cite{conrad2013galois}, that this polynomial generates a quartic extension with Galois group $D_4$, since its cubic resolvent is equal to 
$$X^3+(T+1)X^2+T^2,$$
which is reducible with a unique root, equal to $T$, while its quadratic resolvent is equal to 
$$X^2+T^2X+(T^5+T^3+T^2),$$ 
which is irreducible over $\F_2(T)$. Hence, the only non trivial subextension of $F$ is $\F_4(T)$.

A similar computation shows that the quartic extension $K/\F_4(T)$ has group $D_4$ too.  Hence $K$ contains only $F$ as non trivial extension of $\F_4(T)$, and this fact together with the previous one shows that $K$ is a field of degree 8 over $\F_2(T)$ with only $\F_4(T)$ and $F$ as subfields. Notice that from this we derive $F=\F_2(T)(\alpha)$ and $K=\F_2(T)(\beta)$.
\\\\
\textit{3) We describe explicitly the Galois closure of $F$.} Let us give some notation: we denote by $N$ the Galois closure of $K/\F_2(T)$ and by $L$ the Galois closure of $F/\F_2(T)$. We already know that $\Gal(L/\F_2(T))\simeq D_4$: in this part of the proof we want to explicitly describe $L$.

In order to do this, we need to describe completely the Galois orbit of $\alpha$: this coincides with the set of roots of the polynomial defining $F$, i.e. it is formed by the roots of the quadratic polynomial defining $\alpha$ and those of the conjugate polynomial. If $\alpha_T$ is a fixed root of $X^2+X+T$, which is irreducible over $\F_2(T)$, these roots are
$$
\{\alpha,\alpha+1,\alpha+\alpha_T,\alpha+\alpha_T+1\}.
$$
The first two numbers are the roots of $X^2+X+\zeta T$, while the remaining ones are the roots of the conjugate $X^2+X+(\zeta+1)T$: here and thereafter we will extensively use the property that, over a field with characteristic 2, if $r_1$ is a root of an Artin-Schreier polynomial $X^2+X+a_1$ and $r_2$ is a root of $X^2+X+a_2$, then $r_1+r_2$ is a root of $X^2+X+(a_1+a_2)$. Notice that choosing $\alpha_T+1$ instead of $\alpha_T$ does not change the description of the Galois orbit $\alpha$.

Hence we have $L=\F_2(T)(\alpha,\alpha_T)$ and the Galois group, isomorphic to $D_4$, is generated by the following automorphisms of $L$: 
$$
\tau:\begin{cases}
\alpha \to \alpha\\
\zeta\to\zeta\\
\alpha_T\to\alpha_T+1
\end{cases}
\sigma:\begin{cases}
\alpha\to\alpha+\alpha_T\\
\zeta\to\zeta+1\\
\alpha_T\to\alpha_T+1.
\end{cases}
$$
Notice that the images of these two automorphisms are coherent: if $\zeta$ is fixed, then the only possible choices for the image of $\alpha$ are either $\alpha$ or $\alpha+1$, while sending $\zeta$ to $\zeta+1$ forces $\alpha$ to have image equal to either $\alpha+\alpha_T$ or $\alpha+\alpha_T+1$. In both cases, we can decide indipendently the image of $\alpha_T$ since $F$ and $\F_2(T)(\alpha_T)$ are linearly disjoint over $\F_2(T)$.

The automorphism $\tau$ has order 2 and $L^{\langle\tau\rangle} = F$, while $\sigma$ has order 4 and $L^{\langle\sigma\rangle}=\F_2(T)(\alpha_T+\zeta)$. It is immediate to verify that $\tau\sigma\tau = \sigma^3$ and thus looking at the fixed fields $L^H$ for every subgroup $H$ of $\langle\sigma,\tau\rangle$ we obtain the following subfield lattice for $L$.
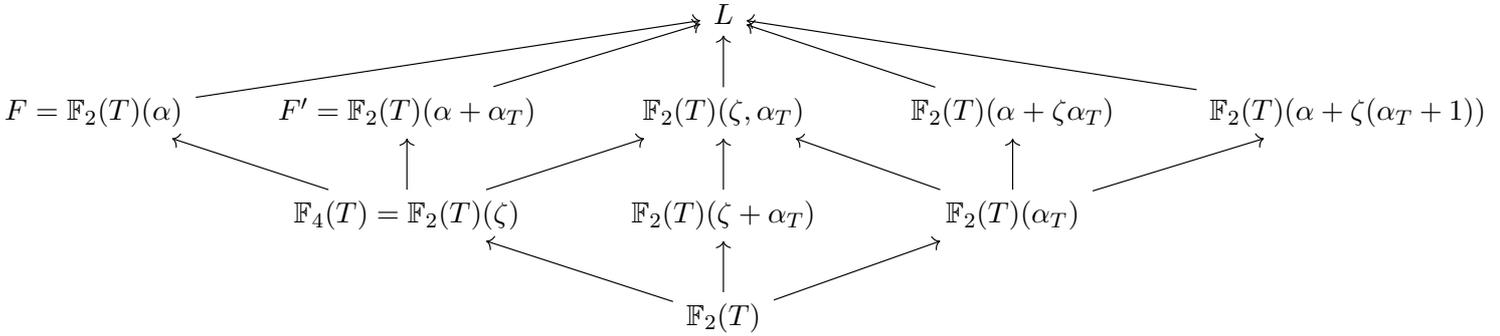
$$\hspace{-2.5cm} 
  \begin{tikzcd}
    & & &L & &\\
    &F=\F_2(T)(\alpha)\arrow{rru} &F'=\F_2(T)(\alpha+\alpha_T)\arrow{ru} & \F_2(T)(\zeta,\alpha_T)\arrow{u} & \F_2(T)(\alpha+\zeta\alpha_T)\arrow{lu} & \F_2(T)(\alpha+\zeta(\alpha_T+1))\arrow{llu}\\
    & &\F_4(T) = \F_2(T)(\zeta)\arrow{lu}\arrow{u}\arrow{ru} & \F_2(T)(\zeta+\alpha_T)\arrow{u} & \F_2(T)(\alpha_T)\arrow{lu}\arrow{u}\arrow{ru} &\\
    & & &\F_2(T)\arrow{lu}\arrow{u}\arrow{ru} & &
    \end{tikzcd}
$$
\noindent
\textit{4) We use the Galois structure of $L$ to compute the Galois orbit of $\beta$.} The defining polynomial of $K/F$, i.e. the polynomial with roots $\beta$ and $\beta+1$, has coefficients in $F$ and hence it is fixed by $\langle\tau\rangle$: in order to describe the conjugates of the polynomial over $\F_2(T)$, it is thus necessary and sufficient to describe its image via the morphisms in the group $\langle\sigma\rangle$, and thus to vary $\alpha$ in its Galois orbit and $\zeta$ according to it. The 4 conjugated polynomials are then
$$X^2+X+(T^2+\zeta T)\alpha,$$
$$X^2+X+(T^2+\zeta T)(\alpha+1),$$
$$X^2+X+(T^2+(\zeta +1)T)(\alpha+\alpha_T),$$
$$X^2+X+(T^2+(\zeta +1) T)(\alpha +\alpha_T+1).$$
The product of these 4 polynomials gives exactly the polynomial \eqref{polK} (this, as other products of polynomials in this proof, can be verified in PARI/GP \cite{pari}). We notice that the first two polynomials define quadratic extensions of $F$, while the remaining two give quadratic extensions of $F'$. By using the same trick on Artin-Schreier polynomials employed before, we can explicitly describe the roots of these polynomials (and thus the Galois orbit of $\beta$) which are
\begin{align*}
    \{&\beta,\hspace{0.1cm}\beta+1,\hspace{0.1cm} \beta+\alpha+\alpha_T+T,\hspace{0.1cm}\beta+\alpha+\alpha_T+T+1,\\
    &\beta+\gamma,\hspace{0.1cm}\beta+\gamma+1,\hspace{0.1cm}\beta+\gamma+\alpha+T,\hspace{0.1cm}\beta+\gamma+\alpha+T+1\}
\end{align*}
where $\gamma$ is a fixed root of the polynomial
\begin{align}\label{polGamma}
    X^2+X+\alpha T + (T^2+(\zeta+1)T)\alpha_T
\end{align}
which has coefficients is the field $ \F_2(T)(\alpha+\zeta\alpha_T)\subset L$. Notice that in the description of the orbit of $\beta$ we needed the root of $X^2+X+T^2$ at some point; since we are working over fields with characteristic 2, one can verify that $\alpha_T+T$ is a root of this polynomial.

We have $\gamma\in N$ by definition of Galois closure, and thus $M\subset N$ where $M$ is the Galois closure of the field $K_{\gamma}\coloneqq\F_2(T)(\alpha+\zeta\alpha_T, \gamma)$ (we shall verify in the next lines that actually $K_{\gamma}=\F_2(T)(\gamma)$). Understanding the structure of $M$ and detecting the conjugates of $\gamma$ are crucial steps for the proof.\\\\
\textit{5) We study the Galois orbit of $\gamma$ and the Galois structure of $M$.}
The element $\gamma$ has defining polynomial in $\F_2(T)(\alpha+\zeta\alpha_T) = L^{\langle\sigma\tau\rangle}$: thus the conjugates of the polynomial \eqref{polGamma} are obtained applying to it representatives of the left cosets of $\langle\sigma\tau\rangle$ in $\langle\sigma,\tau\rangle$, and we obtain the 4 polynomials
$$X^2+X+\alpha T + (T^2+(\zeta+1)T)\alpha_T,$$
$$X^2+X+\alpha T + (T^2+(\zeta+1)T)\alpha_T+T,$$
$$X^2+X+\alpha T + (T^2+(\zeta+1)T)\alpha_T+(T^2+(\zeta+1)T),$$
$$X^2+X+\alpha T + (T^2+(\zeta+1)T)\alpha_T+(T^2+\zeta T).$$
The element $\gamma$ thus has degree 8 over $\F_2(T)$, so that $K_{\gamma}=\F_2(T)(\gamma)$, and its Galois orbit is equal to
\begin{align*}
    \{&\gamma,\hspace{0.1cm}\gamma+1,\hspace{0.1cm} \gamma+\alpha_T,\hspace{0.1cm}\gamma+\alpha_T+1,\\
    &\gamma+\alpha+T,\hspace{0.1cm}\gamma+\alpha+T+1,\hspace{0.1cm}\gamma+\alpha_T+\alpha+T,\hspace{0.1cm}\gamma+\alpha_T+\alpha+T+1\}.
\end{align*}
Notice that $\alpha_T\in \F_2(T)(\alpha+\zeta\alpha_T)$ and thus $\gamma+\alpha_T$ actually belongs to $K_{\gamma}$, so that this field only has a single different conjugated field and its Galois closure is $M=K_{\gamma}(\alpha)$.

The product of the 4 polynomials above gives a defining polynomial for $K_{\gamma}$ over $\F_2(T)$, equal to

\begin{equation}\label{polgamma}
    X^8 + (T^4+T^3+1)X^4+T^5X^2+(T^5+T^4+T^3)X+(T^{10}+T^8+T^5).
\end{equation}
We can compute the Galois group of $K_{\gamma}$ by giving this polynomial as input in Magma \cite{cannon2011handbook}: the group of $K_{\gamma}$ results to be isomorphic to $D_4\times C_2$, so that $M$ is a quadratic extension of both $K_{\gamma}$ and $L$ and it has degree 16 over $\F_2(T)$. We also determine all quadratic subfields of $K_{\gamma}$ via Magma, and we find there are 3 of them, generated respectively by the elements $\alpha_T, \alpha_{T^3}$ and $\alpha_T+\alpha_{T^3}$, where $\alpha_{T^3}$ is a chosen root of $X^2+X+T^3$. Since $L=\F_2(T)(\alpha,\alpha_T)$, we have $M=L(\alpha_{T^3})=\F_2(T)(\alpha,\alpha_T,\alpha_{T^3})=\F_2(T)(\alpha,\gamma)$: an explicit description of the group of $M$ is
$$\langle\sigma,\tau,\omega | \sigma^4=\tau^2=\omega^2=1, \sigma\omega=\omega\sigma, \tau\omega=\omega\tau, \tau\sigma\tau=\sigma^3\rangle$$
where
$$
\tau:\begin{cases}
\alpha \to \alpha\\
\zeta\to\zeta\\
\alpha_T\to\alpha_T+1\\
\alpha_{T^3}\to\alpha_{T^3}
\end{cases}
\sigma:\begin{cases}
\alpha\to\alpha+\alpha_T\\
\zeta\to\zeta+1\\
\alpha_T\to\alpha_T+1\\
\alpha_{T^3}\to\alpha_{T^3}
\end{cases}
\omega:\begin{cases}
\alpha\to\alpha\\
\zeta\to\zeta\\
\alpha_T\to\alpha_T\\
\alpha_{T^3}\to\alpha_{T^3}+1.
\end{cases}
$$
\\\\
\textit{6) We show that $N=\F_2(T)(\beta,\alpha_T,\alpha_{T^3})$ has degree 32.} From the description of the Galois orbit of $\beta$ we already know that $N=\F_2(T)(\beta,\gamma).$ We want to prove that $N/M$ is a quadratic extension: in order to do this, it is enough to show that $K\not\subset M$, since in this case $N=KM$, $F=K\cap M$ and
$$
16 < 16a= [KM:\F_2(T)] \leq \frac{[K:\F_2(T)][M:\F_2(T)]}{[K\cap M:\F_2(T)]} = 32,
$$
forcing $[KM:\F_2(T)]=32$.

Now, if $K\subset M$, then $K=M^{\langle g\rangle}$ for some $g\in\Gal(M/\F_2(T))$ of order 2: since $g$ fixes $K$, it must also fix $\alpha$ and $\zeta$. The only  automorphisms $g$ with this property are $\tau, \omega$ and $\tau\omega$: however, if $g=\tau$, then $\alpha_{T^3}\in K$ since $\tau$ fixes this element, and this is not possible because we know that the only quadratic subfield of $K$ is $\F_4(T)$. Similarly for the other two choices: we would have $\alpha_{T^3}\in K$ or $\alpha_{T}+\alpha_{T^3}\in K$, which is impossible.
\\\\
\textit{7) We prove that $\Gal(N/\F_2(T))\simeq C_8\rtimes V_4$.} Since $N=\F_2(T)(\beta,\alpha_T,\alpha_{T^3})$, every automorphism of $N$ is uniquely defined by the images of the elements $\beta,\alpha_T$ and $\alpha_{T^3}$: these images can be chosen independently from each other since we obtained that these are linearly disjoint over $\F_2(T)$ as necessary step for the description of $N$. The choice of these three values will give immediately the images for $\alpha, \zeta$ and $\gamma$.

We explicitly describe the Galois group of $N$ by considering a suitable biquadratic subextension. Let $B\coloneqq \F_2(T)(\alpha_T+\zeta,\alpha_{T^3})$: we have $K\cap B=\F_2(T)$ and  $N=KB$ (the choice of $\alpha_T+\zeta$ instead of $\alpha_T$ will be justified in the next lines). Consider the automorphism of $N$ defined by the choice
$$
\varphi_1:\begin{cases}
\beta \to \beta +\gamma +\alpha + T\\
\alpha_T+\zeta \to \alpha_T+\zeta\\
\alpha_{T^3}\to\alpha_{T^3}.
\end{cases}
$$
Then $\varphi_1\in\Gal(N/B)$, which is a group of order 8, and these conditions uniquely define the images of $\alpha_T,\zeta,\alpha$ and $\gamma$ via $\varphi_1$: in fact, we have
$$
\varphi_1:\begin{cases}
\beta \to \beta +\gamma +\alpha + T\\
\alpha\to\alpha+\alpha_T+1\\
\zeta\to\zeta+1\\
\alpha_{T^3}\to\alpha_{T^3}\\
\alpha_T\to\alpha_T+1\\
\gamma\to\gamma+\alpha+T+\delta
\end{cases}
$$
where $\delta\in\{0,1\}$. The images of $\alpha$ and $\zeta$ derive from the choice of the image of $\beta$ (which means that we have chosen the last of the conjugates of its polynomial) and the image of $\alpha_T$ follows in order to respect the constraint on $\alpha_T+\zeta$. For what concerns $\gamma$, the choice for $\alpha, \zeta$ and $\alpha_T$ gives one among $\gamma+\alpha+T$ and $\gamma+\alpha+T+1$, and the choice of $\alpha_{T^3}$ uniquely defines one of these. Since we are not able to explicitly describe this choice, we say that the image of $\gamma$ is $\gamma+\alpha+T+\delta$, with $\delta$ uniquely determined by $\alpha_{T^3}\to\alpha_{T^3}$.

For $m\in\Z$, we denote $\varphi_m\coloneqq\underbrace{\varphi_1\circ\cdots\circ\varphi_1}_m$. We have that the orbit of $\beta$ via $\varphi_1$ is cyclic of order 8, since
$$
\begin{matrix}
    &\varphi_1(\beta) = \beta+\gamma+\alpha+T,& \varphi_2(\beta)=\beta+\alpha+\alpha_T+T+\delta+1,\\
    &\varphi_3(\beta)=\beta+\gamma+1+\delta,& \varphi_4(\beta)=\beta+1,\\
    &\varphi_5(\beta) = \beta+\gamma+\alpha+T+1, & \varphi_6(\beta)=\beta+\alpha+\alpha_T+T+\delta,\\
    &\varphi_7(\beta)=\beta+\gamma+\delta, &\varphi_8(\beta)=\beta.
\end{matrix}
$$
Similarly, one can verify that the orbits of $\gamma$ and $\alpha$ are cyclic of order 4, while for the remaining elements they are of order 2 or 1. Thus $\varphi_1$ is an automorphism of order 8, and so $A\coloneqq\Gal(N/B)=\langle\varphi_1\rangle$ is cyclic of order 8. 

Now, we consider the group $H\coloneqq\Gal(N/K)$ of order 4. Every automorphism in this group fixes $\beta,\alpha$ and $\zeta$ and it is uniquely determined by the images of $\alpha_T$ and $\alpha_{T^3}$, so that it is immediate to see that $H\simeq V_4$. We give now an explicit description of the non trivial elements of $H$ (the identity element will be denoted as $\psi_1$).

$$
\psi_3:\begin{cases}
\alpha_T\to\alpha_T+1,\\
\alpha_{T^3}\to\alpha_{T^3}+1\\
\beta\to\beta\\
\alpha\to\alpha\\
\zeta\to\zeta\\
\gamma\to\gamma+\alpha+T+\delta+1
\end{cases}
\psi_5:\begin{cases}
\alpha_T\to\alpha_T,\\
\alpha_{T^3}\to\alpha_{T^3}+1\\
\beta\to\beta\\
\alpha\to\alpha\\
\zeta\to\zeta\\
\gamma\to\gamma+1
\end{cases}
\psi_7:\begin{cases}
\alpha_T\to\alpha_T+1,\\
\alpha_{T^3}\to\alpha_{T^3}\\
\beta\to\beta\\
\alpha\to\alpha\\
\zeta\to\zeta\\
\gamma\to\gamma+\alpha+T+\delta
\end{cases}
$$
Notice that the choice of the image of $\gamma$ via $\psi_7$ is coherent with the image of $\gamma$ via $\varphi_1$, since the images of $\alpha_T, \alpha$ and $\zeta$ detect the same conjugate of the polynomial \eqref{polgamma} and the constraint on $\alpha_{T^3}$ must give the same image.

One can verify that these automorphisms satisfy the relation $\psi_l \varphi_1 \psi_l = \varphi_l$ and so we have a group morphism
\begin{align*}
    &H \rightarrow \text{Aut}(A)\\
    &\psi_l \to (\varphi_1\to\varphi_l)
\end{align*}
which is exactly the semidirect product structure $(\Z/8\Z)\rtimes (\Z/8\Z)^*$ given by the isomorphism $(\Z/8\Z)^* \simeq \text{Aut}(\Z/8\Z)$. Hence, we have obtained that $\Gal(N/\F_2(T)) = A\rtimes H$ is isomorphic to the group 8T15. This proves that $K$ has an arithmetically equivalent sibling contained in $N$, since $N$ contains 19 subfields of degree 8 among which 11 of them are either Galois of degree 8 or with Galois closure of degree 16, and the remaining 8 ones form 2 isomorphism classes, each one containing 4 fields, and every field in the first class is equivalent to a field in the second class (these properties can be recovered by a number field example: compute in PARI/GP the Galois closure of the number field given by $X^8-3$ and then compute its subfields of degree 8 and their Galois groups via suitable PARI/GP commands).\\\\
\textit{8) We explicitly detect the sibling $K'$ of $K$.} Let $G\coloneqq \Gal(N/\F_2(T))$: in the previous lines we showed that $G=A\rtimes H$ with $A=\Gal(N/B)$ and $H=\Gal(N/K)$. Let $\sigma: H\to G$ be the section of $H$ corresponding to $K$, i.e. such that $N^{\sigma(H)}=K$: we will find the arithmetically equivalent sibling of $K$ using the same procedure employed by Perlis, 
i.e. the one described in Lemma \ref{lemmaCohomologic}. We need to find a non-trivial element $[\chi]$ of the cohomology group $H^1(H,A)$ such that for every $h\in H$ its restriction to the cohomology groups $H^1(\langle h\rangle,A)$ is trivial.

Let $\chi: H \to A$ be the function defined as 
$$
\begin{matrix}
    &\chi(\psi_1) = id =\chi(\psi_7),\\
    &\chi(\psi_3) = \varphi_4 = \chi(\psi_5).
\end{matrix}
$$
It is easy to verify that this is a 1-cocycle, i.e. it satisfies the relation $\chi(\sigma\tau) = \chi(\sigma)^{\tau}\chi(\tau)$ for every $\sigma,\tau\in H$ (where $\varphi_m^{\psi_l} = \varphi_{lm}$). At the same time, one can verify that $\chi$ is not a 1-coboundary, i.e. there is not $\varphi_m\in A$ such that $\chi(\sigma) = \varphi_m^{\sigma}\cdot\varphi_{-m}$. Hence $[\chi]\in H^1(H,A)$ is non-trivial: however, every time $\chi$ is restricted to a subgroup $\langle h\rangle$ with $h\in H$, its restriction becomes a 1-coboundary (we have $\varphi_m=id$ for $h\in\{\psi_1,\psi_7\}$, $\varphi_m=\varphi_2$ for $h=\psi_3$ and $\varphi_m=\varphi_3$ for $h=\psi_5$) so that $[\chi]=[0]$ in every $H^1(\langle h\rangle,A)$. By Lemma \ref{lemmaCohomologic}, the section $(\chi\cdot\sigma): H\to G$ gives a group which is Gassmann equivalent to $\sigma(H)$ but not conjugated to it in $G$.

We show that $\beta+\alpha_{T^3}\in N$ is fixed by the group $(\chi\cdot\sigma)(H)$. For every $h\in H$, we have
\begin{align*}
    (\chi\cdot \sigma)(h)(\beta+\alpha_{T^3}) = \chi(h)(\sigma(h)(\beta+\alpha_{T^3})) = \chi(h)(\beta+\sigma(h)(\alpha_{T^3})) 
\end{align*}
because $\sigma(h)(\beta)=\beta$ by definition. Now, if $h=\psi_1$ or $\psi_7$, we have 
\begin{align*}
    \chi(h)(\beta+\sigma(h)(\alpha_{T^3})) = id(\beta+\alpha_{T^3})=\beta+\alpha_{T^3}
\end{align*}
while, if $h=\psi_3$ or $\psi_5$, we have
\begin{align*}
    \chi(h)(\beta+\sigma(h)(\alpha_{T^3})) = \varphi_4(\beta+\alpha_{T^3}+1)=(\beta+1)+(\alpha_{T^3}+1)=\beta+\alpha_{T^3}.
\end{align*}
We finally show that $\beta+\alpha_{T^3}$ is an element of order 8: this is a root of the polynomial\\
$X^2+X+(T^2+\zeta T)\alpha+T^3$, for which we have the tower of fields
$$
\begin{tikzcd}
&K'& X^2+X+(T^2+\zeta T)\alpha+T^3& (\beta+\alpha_{T^3})\\
&F\arrow{u}& X^2+X+\zeta T& (\alpha)\\
&\F_{4}(T)\arrow{u}&X^2+X+1 &(\zeta)\\
&\F_2(T).\arrow{u}&
\end{tikzcd}
$$
For this tower we can verify all the properties already proved for $K$: in particular, the tower above coincides with the subfield lattice of $K'$, which is a field of degree 8 over $\F_2(T)$, and $K'=\F_2(T)(\beta+\alpha_{T^3})$. Moreover, computing all the conjugates of its polynomial and multiplying them we obtain its defining polynomial over $K$, which is exactly \eqref{polL}.
\end{proof}

Following the very same sketch of this proof, it is actually possible to describe infinite couples of non-geometrically equivalent siblings.

\begin{thm}
Let $p\geq 2$ prime. The two  extensions $K(p)$ and $K'(p)$ of $\F_2(T)$ given by the polynomials
\begin{align}\label{polKp}
    X^8 &+ TX^6 + TX^5 +(T^{2p} + T^{p+2} + T^{p+1} + T^5 + T^4 + T^3 + T^2 + T+1)X^4 + TX^3 \nonumber\\
    &+ (T^{2p+2}+T^{2p+1}+T^{2p}+T^{p+4}+T^{p+3}+T^{p+1}+T^7 + T^5  + T^3 + T^2 )X^2 \nonumber\\&+ (T^{2p+2}+T^{2p+1}+T^{p+4}+T^{p+3}+T^{p+2}+T^7 + T^4)X \nonumber\\&+ (T^{4p}+T^{3p+2}+T^{2p+5}+T^{2p+4}+T^{2p+2}+T^{p+6}+T^{p+5}+T^{p+4}+T^{10} + T^8 +T^6)
\end{align}
and
\begin{align}\label{polLp}
    X^8 &+ TX^6 + TX^5 + (T^{2p}+T^{p+2}+T^{p+1}+T^6 + T^3 + T^2 + T + 1)X^4 + TX^3 \nonumber\\
    &+ (T^{2p+2}+T^{2p+1}+T^{2p}+T^{p+4}+T^{p+3}+T^{p+1}+T^8 + T^7 + T^5 + T^4 + T^3 + T^2)X^2 \nonumber\\&+ (T^{2p+2}+T^{2p+1}+T^{p+4}+T^{p+3}+T^{p+2}+T^8 + T^7 + T^6 + T^5 + T^4 )X \nonumber\\&+ (T^{4p}+T^{3p+2}+T^{2p+4}+T^{2p+2}+T^{p+8}+T^{p+6}+T^{p+5}+T^{p+4}+T^{12} + T^9 + T^8 + T^7 + T^6)\nonumber\\
    &
\end{align}
are non-geometrically arithmetically equivalent siblings over $\F_2(T)$ with group 8T15 such that they are not equivalent over $\F_4(T)$. Moreover, $K(2)\simeq K$ and $K'(2)\simeq K'$; $K(3)\simeq K'$ and $K'(3)\simeq K$; if $p\neq q$ and $p\geq 5$, then $K(p)$ is not equivalent to $K(q)$.
\end{thm}

\begin{proof}
The field $K(p)$ is the field of degree 8 over $\F_2(T)$ with primitive element $\beta+\alpha_{T^p}$, where $\alpha_{T^p}$ is a fixed root of $X^2+X+T^p$, and is a root of the polynomial $X^2+X+(T^2+\zeta T)\alpha+T^p$. Similarly, the field $K'(p)$ has degree 8 over $\F_2(T)$ and primitive element $\beta+\alpha_{T^p}+\alpha_{T^3}$ which is a root of  $X^2+X+(T^2+\zeta T)\alpha+T^p+T^3$.

For every value of $p$, the proof that $K(p)$ and $K'(p)$ are non-geometric arithmetically equivalent siblings is exactly the same as for $K$ and $K'$: the crucial key is that their Galois closure $N(p)$ contains the very same field $K_{\gamma}$ described in the previous proof, so that the Galois orbit of $\beta+\alpha_{T^p}$ is obtained by the one of $\beta$ by replacing $\beta$ with $\beta+\alpha_{T^p}$. In particular, $N(p)=\F_2(T)(\beta+\alpha_{T^p},\alpha_T,\alpha_{T^3})$ and contains the same Galois field $M=\F_2(T)(\alpha,\alpha_T,\alpha_{T^3})$ of degree 16.

If $p=2$, then $\beta+\alpha_{T^2} = \beta+\alpha_T+T$ is exactly in the same conjugate field of $K$ containing $\beta+\alpha+\alpha_T+T$, and so $K(2)\simeq K$ (similarly for $K'(2)$). 

If $p=3$, then the primitive element of $K(3)$ is exactly $\beta+\alpha_{T^3}$, i.e. the primitive element of $K'$, and the primitive element of $K'(3)$ is exactly $\beta$, so that $K(3)=K'$ and $K'(3)=K$. 

If $p\neq q$ and $p\geq 5$ we cannot have $\beta+\alpha_{T^p}\in N(q)$ because otherwise $\alpha_{T^p}+\alpha_{T^q}\in N(q)$, which is not possible: the Galois closure $N(q)$ of $K(q)$ contains 7 quadratic subfields generated by $\zeta, \alpha_T, \alpha_{T^3}$ and sums of these elements. This shows that $K(p)$ and $K(q)$ are not equivalent.  
\end{proof}

We conclude this section by showing that every couple $(K(p),K'(p))$ is arithmetically equivalent with no exceptions, and so they share the same lifted Goss zeta function. First, we need the following lemma.

\begin{lem}
Let $K/F$ be a quadratic extension of global function fields over $\F_2(T)$. Then $K/F$ is unramified at every finite place.
\end{lem}

\begin{proof}
First of all, since $K=F(\alpha)$ with $\alpha$ being a root of an Artin-Schreier polynomial $g(X)=X^2+X+c$ with $c\in F$, we show that $\Ok=\Of[\alpha]$. Let $\beta\in \Ok$: then $\beta=a+b\alpha$ with $a,b\in F$, and we want to show that $a,b\in\Of$. The trace of $\beta$ over $F$ belongs to $\Of$ and is equal to $b$, so that $b\in\Of$: but then $a=\beta-b\alpha$ is an integral element over $\F_2(T)$ which belongs to $F$, so that $a\in\Of$.

Now, we show that $K/F$ is unramified at every finite place: this is equivalent to showing that the different ideal of the extension is trivial. Since $\Ok=\Of[\alpha]$, the different ideal is the principal ideal generated by $g'(\alpha)$: but $g'(X)=1$ since we are considering fields with characteristic 2, and thus the ideal generated by $g'(\alpha)$ is trivial.
\end{proof}

\begin{coro}
Let $K_1\subset K_2\subset\cdots\subset K_n$ be a tower of quadratic extensions of global function fields over $\F_2(T)$. Then $K_n/K_1$ is unramified at finite places.
\end{coro}

\begin{coro}
Let $K(p)/\F_2(T)$ be defined by the polynomial \eqref{polKp}. Then its Galois closure $N(p)/\F_2(T)$ is unramified at finite places. 
\end{coro}

\begin{proof}
We have the tower of quadratic extensions $N(p)/M/K_{\gamma}/\F_2(T)(\alpha+\zeta\alpha_T)/\F_2(T)(\alpha_T)/\F_2(T)$.
\end{proof}

\begin{coro}
The extensions $K(p)/\F_2(T)$ and $K'(p)/\F_2(T)$ have the same lifted Goss zeta function.
\end{coro}

\begin{proof}
The two fields are arithmetically equivalent over $\F_2(T)$ and contained in the unramified (at finite places) extension $N(p)/\F_2(T)$, so they are both unramified at finite places over $\F_2(T)$. Thus they are equivalent over $\F_2(T)$ with no exceptions, and so they have the same lifted Goss zeta function by Theorem \ref{ThmAnalyticEquiv}.
\end{proof}

We conclude this section by giving an example of a finite prime of $\F_4(T)$ with different splitting type in $K$ and $K'$, showing explicitly that the two fields are not equivalent over $\F_4(T)$.

Let $f(T)\coloneqq T^8+T^6+T^5+T^3+1$: this polynomial is irreducible in $\F_2[T]$ and decomposes as product of two irreducible polynomials of degree 4 in $\F_4[T].$ The two factors can be detected in the following way: since the extension is unramified, we know that if $(f(T))\F_4[T] = \mathfrak{p}\cdot\mathfrak{q}$, then $\mathfrak{p} = (f(T), \zeta + g(T))$ and $\mathfrak{q}=(f(T), \zeta+g(T)+1)$, where $g(T)$ is a polynomial in $\F_2[T]$ with degree $\leq 7$ such that $x^2+x+1 = (x+g(T))(x+g(T)+1)$ in the quotient field $k\coloneqq \F_2[T]/(f(T))$. Since the degree of $f(T)$ is even, this field always contain an isomorphic copy of $\F_4$, and thus $x^2+x+1$ is reducible: the element $g(T)$ is a chosen third root of unit in $k$.

One can verify that, for this choice of $k$, the polynomial $T^6+T^5+T^4+T^3+T^2+T$ is a primitive third root of unity. Hence we have
$$\mathfrak{p} = (T^8+T^6+T^5+T^3+1,\text{ } T^6+T^5+T^4+T^3+T^2+T+\zeta),$$
$$\mathfrak{q} = (T^8+T^6+T^5+T^3+1,\text{ } T^6+T^5+T^4+T^3+T^2+T+1+\zeta).$$
These are prime ideals in $\F_4[T]$, which is a PID: a unique generator can be found by applying Euclidean division to the two generators of the ideal, and these generators are $T^4+\zeta T^2+(\zeta+1)T+(\zeta+1)$ for $\mathfrak{p}$ and $T^4+(\zeta+1) T^2+\zeta T+\zeta$.

\begin{thm}
The prime $\mathfrak{p}$ has different splitting types in $K$ and $K'$.
\end{thm}

\begin{proof}
Let us first study how $\mathfrak{p}$ decomposes in $F$. We must see how the polynomial $X^2+X+\zeta T$ decomposes in the quotient field $\F_4[T]/\mathfrak{p}$: this field is
$$\frac{\F_4[T]}{(T^8+T^6+T^5+T^3+1, \text{ }T^6+T^5+T^4+T^3+T^2+T+\zeta)}$$
and is thus identified with $k$ under the assumption that $\zeta$ is equal to $T^6+T^5+T^4+T^3+T^2+T$. Therefore we just want to study how the polynomial
$$X^2+X+(T^6+T^5+T^4+T^3+T^2+T)T = X^2+X+T^7+T^6+T^5+T^4+T^3+T^2$$
decomposes in $k$. One finds that is equal to
$$(X+T^7+T^5+T^4+T)(X+T^7+T^5+T^4+T+1)$$
and therefore $\mathfrak{p}$ splits in $F$ as product of the prime ideals
$$\mathfrak{p}_1\coloneqq (T^8+T^6+T^5+T^3+1,\text{ } T^6+T^5+T^4+T^3+T^2+T+\zeta,\text{ }T^7+T^5+T^4+T+\alpha),$$
$$\mathfrak{p}_2\coloneqq (T^8+T^6+T^5+T^3+1,\text{ } T^6+T^5+T^4+T^3+T^2+T+\zeta,\text{ }T^7+T^5+T^4+T+1+\alpha).$$
These primes have inertia 1 over $\F_4(T)$ and over $\F_2(T)$, so that when we study their decomposition in $K$ and $K'$ we just look for the factorization of the defining polynomials in $k$ after identifying $\alpha$ and $\zeta$ with the corresponding polynomials in the ideals.

We look at their decomposition in the extension $K/F$, which is defined by the polynomial $X^2+X+(T^2+\zeta T)\alpha$. Let us begin with $\mathfrak{p}_1$ (meaning that $\zeta\to T^6+T^5+T^4+T^3+T^2+T$ and $\alpha\to T^7+T^5+T^4+T$). The polynomial becomes $X^2+X+T^7+T^5+T^4+T^3+T^2+T$, and one can verify that it has no roots in $k$, implying that $\mathfrak{p}_1$ is inert in $K$. For what concerns $\mathfrak{p}_2$, the same procedure with $\alpha+1$ instead of $\alpha$ gives $X^2+X+T^6+T^2+T$, which again has no roots in $k$. Therefore the two primes remain inert in $K$.

Looking instead at their decomposition in $K'/F$, the defining polynomial is $X^2+X+(T^2+\zeta T)\alpha+T^3$, and so we just need to add $T^3$ to the previously obtained polynomials with coefficients in $k$. But now, the two new polynomials, $X^2+X+T^7+T^5+T^4+T^2+T$ for $\mathfrak{p}_1$ and $X^2+X+T^6+T^3+T^2+T$ for $\mathfrak{p}_2$, split in $k$.

This means that $\mathfrak{p}$ has splitting type $(2,2)$ in $K$ and $(1,1,1,1)$ in $K'$. One can verify, following the same steps, that $\mathfrak{q}$ has splitting type $(1,1,1,1)$ in $K$ and $(2,2)$ in $K'$, so that in the end the ideal $(f(T))$ has splitting type $(1,1,1,1,2,2)$ in both $K$ and $K'$ (as we expected).
\end{proof}


\begin{thebibliography}{CKVdZ10}

\bibitem[BdS02]{bosma2002arithmetically}
Wieb Bosma and Bart de~Smit.
\newblock On arithmetically equivalent number fields of small degree.
\newblock In {\em International Algorithmic Number Theory Symposium}, pages
  67--79. Springer, 2002.

\bibitem[BM83]{butler1983transitive}
Gregory Butler and John McKay.
\newblock The transitive groups of degree up to eleven.
\newblock {\em Communications in Algebra}, 11(8):863--911, 1983.

\bibitem[Bou07]{bourbaki2007topologie}
Nicolas Bourbaki.
\newblock {\em Topologie g{\'e}n{\'e}rale: Chapitres 1 {\`a} 4}, volume~3.
\newblock Springer Science \& Business Media, 2007.

\bibitem[CBFS11]{cannon2011handbook}
John Cannon, Wieb Bosma, Claus Fieker, and Allan Steel.
\newblock Handbook of magma functions.
\newblock 2011.

\bibitem[CKVdZ10]{cornelissen2010arithmetic}
Gunther Cornelissen, Aristides Kontogeorgis, and Lotte Van~der Zalm.
\newblock Arithmetic equivalence for function fields, the goss zeta function
  and a generalisation.
\newblock {\em Journal of Number Theory}, 130(4):1000--1012, 2010.

\bibitem[Con13]{conrad2013galois}
Keith Conrad.
\newblock Galois groups of cubics and quartics in all characteristics.
\newblock {\em Unpublished note}, 2013.

\bibitem[dSP94]{desmitperlis1994zeta}
Bart de~Smit and Robert Perlis.
\newblock Zeta functions do not determine class numbers.
\newblock {\em Bulletin of the American Mathematical Society}, 31(2):213--215,
  1994.

\bibitem[FJ06]{fried2006field}
Michael~D Fried and Moshe Jarden.
\newblock {\em Field arithmetic}, volume~11.
\newblock Springer Science \& Business Media, 2006.

\bibitem[FOdf]{fontaine2008theory}
Jean-Marc Fontaine and Yi~Ouyang.
\newblock Theory of p-adic galois representations.
\newblock 2008, available at
  https://www.imo.universite-paris-saclay.fr/~fontaine/galoisrep.pdf.

\bibitem[Gas26]{gassmann1926bemerkungen}
Fritz Gassmann.
\newblock Bemerkungen zur vorstehenden arbeit von hurwitz.
\newblock {\em Math. z}, 25:665--675, 1926.

\bibitem[Gos12]{goss2012basic}
David Goss.
\newblock {\em Basic structures of function field arithmetic}.
\newblock Springer Science \& Business Media, 2012.

\bibitem[Iwa53]{iwasawa1953rings}
Kenkichi Iwasawa.
\newblock On the rings of valuation vectors.
\newblock {\em Annals of Mathematics}, pages 331--356, 1953.

\bibitem[Kli98]{klingen1998arithmetical}
Norbert Klingen.
\newblock {\em Arithmetical similarities: Prime decomposition and finite group
  theory}.
\newblock Oxford University Press, 1998.

\bibitem[KM]{klunersMalle}
J\"{u}rgen Kl\"{u}ners and Gunter Malle.
\newblock A database for number fields.
\newblock Available at \url{http://galoisdb.math.upb.de/home}.

\bibitem[Kom84]{komatsu1984adele}
Keiichi Komatsu.
\newblock On adele rings of arithmetically equivalent fields.
\newblock {\em Acta Arithmetica}, 2(43):93--95, 1984.

\bibitem[KR94]{kani1994idempotent}
Ernst Kani and Michael Rosen.
\newblock Idempotent relations among arithmetic invariants attached to number
  fields and algebraic varieties.
\newblock {\em Journal of Number Theory}, 46(2):230--254, 1994.

\bibitem[{PAR}20]{pari}
{PARI~Group}, Univ. Bordeaux.
\newblock {\em {PARI/GP version {\tt 2.12.0}}}, 2020.
\newblock available at \url{http://pari.math.u-bordeaux.fr/}.

\bibitem[Per77]{perlis1977equation}
Robert Perlis.
\newblock On the equation {$\zeta _{K}(s)=\zeta _{K'}(s)$}.
\newblock {\em Journal of number theory}, 9(3):342--360, 1977.

\bibitem[Sch17]{schneider2017galois}
Peter Schneider.
\newblock {\em Galois Representations and (phi, Gamma)-modules}, volume 164.
\newblock Cambridge University Press, 2017.

\bibitem[Ser65]{serre1965zeta}
Jean-Pierre Serre.
\newblock Zeta and l functions.
\newblock In {\em Arithmetical Algebraic Geometry, Proc. of a Conference held
  at Purdue Univ., Dec. 5-7, 1963}. Harper and Row, 1965.

\bibitem[Ser13]{serre2013local}
Jean-Pierre Serre.
\newblock {\em Local fields}, volume~67.
\newblock Springer Science \& Business Media, 2013.

\bibitem[Sol16]{solomatin2016artin}
Pavel Solomatin.
\newblock On artin l-functions and gassmann equivalence for global function
  fields.
\newblock {\em arXiv preprint arXiv:1610.05600}, 2016.

\bibitem[SP95]{stuart1995new}
Donna Stuart and Robert Perlis.
\newblock A new characterization of arithmetic equivalence.
\newblock {\em Journal of Number theory}, 53(2):300--308, 1995.

\bibitem[Sti09]{stichtenoth2009algebraic}
Henning Stichtenoth.
\newblock {\em Algebraic function fields and codes}, volume 254.
\newblock Springer Science \& Business Media, 2009.

\bibitem[{The}13]{lmfdb}
{The LMFDB Collaboration}.
\newblock The {L}-functions and modular forms database.
\newblock available at \url{http://www.lmfdb.org}, 2013.

\bibitem[Tur78]{turner1978adele}
Stuart Turner.
\newblock Adele rings of global field of positive characteristc.
\newblock {\em Boletim da Sociedade Brasileira de
  Matem{\'a}tica-Bulletin/Brazilian Mathematical Society}, 9(1):89--95, 1978.

\end{thebibliography}
\end{document}